\documentclass[12pt]{iopart}
\newcommand{\be}{\begin{equation*}}
\newcommand{\ee}{\end{equation*}}
\newcommand{\bea}{\begin{eqnarray*}}
\newcommand{\eea}{\end{eqnarray*}}
\newcommand{\pp}{\partial}

\newcommand{\weight}{e^{2s\varphi(x,t)}}

\usepackage{color}

\usepackage{iopams}
\usepackage{amsthm,amssymb,mathrsfs}

\newtheorem{thm}{Theorem}[section]
\newtheorem{lem}[thm]{Lemma}
\theoremstyle{corollary}

\theoremstyle{remark}
\newtheorem{rmk}[thm]{Remark}

\begin{document}

\title[]{Global Lipschitz Stability in Determining Coefficients of 
the Radiative Transport Equation}

\author{Manabu Machida$^1$ and Masahiro Yamamoto$^2$}
\address{$^1$Department of Mathematics, University of Michigan, 
Ann Arbor, MI 48109, USA\\
$^2$Department of Mathematical Sciences, The University of Tokyo,\\
3-8-1 Komaba, Meguro, Tokyo 153, Japan}
\eads{\mailto{mmachida@umich.edu} and \mailto{myama@ms.u-tokyo.ac.jp}}

\begin{abstract}
In this article, for the radiative transport equation, we study
inverse problems of determining a time independent
scattering coefficient or total attenuation by boundary data
on the complementary sub-boundary after making one time input of a pair of
a positive initial value and boundary data on a suitable sub-boundary.
The main results are Lipschitz stability estimates.
We can also prove the reverse inequality, which means that our
estimates for the inverse problems are the best possible.
The proof is based on a Carleman estimate with a linear weight function.
\end{abstract}

%\maketitle

\section{Radiative Transport Equation and Main Results}

We consider wave or particles propagating in a random medium.  
Let $\Omega$ be a bounded domain of $\mathbb{R}^n$, $n\ge 2$ 
with $C^1$-boundary $\partial\Omega$.  The scalar product in 
$\mathbb{R}^n$ is denoted by a dot $(\cdot)$.  
Let $\nabla = \nabla_x = \left(\frac{\pp}{\pp x_1}, ...,
\frac{\pp}{\pp x_n}\right)$.  We let $u(x,v,t)\in\mathbb{R}$ 
denote the angular density at time $t > 0$ and position
$x\in\mathbb{R}^n$ with velocity $v\in V$, where 
$V=\left\{v\in\mathbb{R}^n;\,0<v_0\le|v|\le v_1\right\}$, 
$0\notin\overline{V}$.

Let $\sigma_a(x,v)$ and $\sigma_s(x,v)$ denote the absorption 
and scattering coefficients, respectively.  
Note that $\sigma_a$ and $\sigma_s$ are positive measurable functions: 
$$
\sigma_a:\;\Omega\times V\rightarrow\mathbb{R}, 
\quad
\sigma_s:\;\Omega\times V\rightarrow\mathbb{R}.  \eqno{(1.1)}
$$
We introduce the total attenuation as $\sigma_t=\sigma_a+\sigma_s$.  
The following radiative transport equation, which is the linear 
Boltzmann equation, governs $u(x,v,t)$ for $x\in\Omega$, $v\in V$, 
$0 < t < T$,
$$
Pu :=
P_0u(x,v,t)+\sigma_t(x,v)u- \sigma_s(x,v)\int_V p(x,v,v')u(x,v',t)d{v'}=0,
\eqno{(1.2)}
$$
where
$$
P_0u := \pp_tu(x,v,t) + v\cdot\nabla u(x,v,t).  \eqno{(1.3)}
$$
The phase function $p(x,v,v')$ satisfies
$$
\int_V p(x,v,v')d{v'}=1\quad\mbox{for all}\;(x,v).  \eqno{(1.4)}
$$
Equation (1.2) describes transport in a random medium such as 
light in biological tissue \cite{Arridge99,Arridge09}, 
neutrons in a reactor \cite{Case67, Duderstadt-Martin}, and light in the 
interstellar medium \cite{Chandrasekhar60} and atmospheres \cite{Sobolev75}.  
We let $\nu(x)$ be the outward normal unit vector to $\partial\Omega$ at 
$x\in\partial\Omega$.  We define $\Gamma_{\pm}$ as
$$
\Gamma_+=\left\{(x,v)\in\partial\Omega\times V;\,
\nu(x)\cdot v>0\right\}, \quad
\Gamma_-=\left\{(x,v)\in\partial\Omega\times V;\,
\nu(x)\cdot v<0\right\}.  \eqno{(1.5)}
$$

We impose the following boundary conditions.
$$
u(x,v,0)=a(x,v),\quad x\in\Omega,\quad v\in V,
  \eqno{(1.6)}
$$
$$
u(x,v,t)= g(x,v,t),\quad 0<t<T,\quad (x,v)\in\Gamma_-.  \eqno{(1.7)}
$$

We consider inverse problems of determining $\sigma_t$ or $\sigma_s$ 
by boundary data $u(x,v,t)$, $(x,v) \in \Gamma_+$, $0 < t < T$ 
after setting up the initial value (1.6) and boundary value (1.7) once.  
Our inverse problem is motivated by optical tomography, 
in which we recover $\sigma_t$ and $\sigma_s$ 
from boundary measurements (e.g., \cite{Arridge99,Arridge09}).  
An incident laser beam $g(x,v,t)$ enters 
the sample through the boundary, and the outgoing light 
$u(x,v,t)$ is measured on the boundary.

We refer to works concerning inverse problems on the transport equation.  
Let us write the albedo operator as 
$\mathcal{A}[g]=u(x,v,t)$, $(x,v)\in\Gamma_+$, $0 < t < T$.  
Choulli and Stefanov \cite{Choulli96} proved the uniqueness of $\sigma_t$ 
and $\sigma_s$.  
%\eqno{(1.8)}
%by assuming that the initial value is zero.  
Stability in determining some coefficients among $\sigma_t$, $\sigma_s$, 
$p$ is proved by the angularly averaged albedo operator \cite{Bal09b} 
and by the full albedo operator \cite{Bal10}.  
For the inverse problems in \cite{Bal09b} and \cite{Bal10}, 
the input-output operation can be limited to the boundary and the 
initial value can be zero, but one has to make infinitely many measurements.  
For the stationary transport equation, the non-uniqueness in the coefficient
inverse problem with the albedo
operator was characterized by gauge equivalent pairs in \cite{Stefanov09},
and the Lipschitz stability for gauge equivalent classes was proved for
the time-independent radiative transport equation in \cite{McDowall10}.
See also review articles \cite{Bal09a, Stefanov03} for coefficient inverse 
problems for the radiative transport equation.

Klibanov and Pamyatnykh \cite{Klibanov08} proved the uniqueness of $\sigma_t$ 
by the boundary values of $u$.  The formulation in \cite{Klibanov08} is 
different from \cite{Bal09b}, \cite{Bal10}, \cite{Choulli96} and 
measures a single output on $\Gamma_+\times (0,T)$ 
after choosing initial value and boundary data on $\Gamma_-\times (0,T)$.  

In this article, we adopt the same formulation as in \cite{Klibanov08} 
and we consider the inverse problems of determining 
$\sigma_s$ or $\sigma_t$ by the boundary value on $\Gamma_+ \times (0,T)$ 
with a suitable single input of the initial value.  Our main results are 
Lipschitz stability estimates in determining $\sigma_s$ or $\sigma_t$.  
To the best knowledge of the authors, there are no publications on the 
Lipschitz stability with a single measurement data related to the 
initial/boundary value problem (1.2), (1.6) and (1.7). 
The key of our proof is that we need not any extension of the solution 
$u$ to $(-T,T)$, thanks to the Carleman estimate Lemma 3.2 below.
On the other hand,  \cite{Klibanov08} applies the extension of the solution 
$u$ to $(-T,T)$ and so requires extra conditions for unknown coefficients.

Bukhgeim and Klibanov \cite{BK} proposed a methodology for proving 
the uniqueness and the stability for coefficient inverse problems 
with a single 
measurement, on which \cite{Klibanov08} is based.  
Their method uses an $L^2$-weighted estimate called a Carleman 
estimate for solutions to the differential equation under consideration.  
The Carleman estimate dates back to Carleman \cite{Ca}. See 
H\"ormander \cite{Hormander63}, Isakov \cite{Isa2}, and 
Lavrent'ev, Romanov, and Shishat$\cdot$ski\u{i} \cite{Lavrentev86}.  
As for inverse problems by Carleman estimate, we refer for example to 
Imanuvilov and Yamamoto \cite{ImaY1}, \cite{ImaY2}, Isakov \cite{Isa1}, 
\cite{Isa3}, Klibanov \cite{Kl1}, \cite{Kl2}, Klibanov and Timonov \cite{KT}, 
and Yamamoto \cite{Y}.  
Moreover see Klibanov and Pamyatnykh \cite{Klibanov06} for 
the Carleman estimate for a transport equation and an application to 
the unique continuation, and Klibanov and Yamamoto \cite{KY} for the 
exact controllability for the transport equation.  
Prilepkov and Ivankov \cite{PI} discusses an inverse problem of 
determining a $t$-function in the case where $\sigma_t$ depends on $x,v,t$.  

Throughout this article, $H^m(\Omega)$ denotes usual Sobolev spaces.  
We set
$$
X = H^1(0,T;L^{\infty}(\Omega\times V))
\cap H^2(0,T;L^2(\Omega\times V)).
$$
For arbitrarily fixed constant $M>0$, we set
$$
\mathcal{U} = \{ u\in X;\thinspace
\Vert u\Vert_X + \Vert \nabla u\Vert_{H^1(0,T;L^2
(\Omega\times V))} \le M\}.  \eqno{(1.9)}
$$
Now we are ready to state our main results.
\\

\begin{thm}[Determination of $\sigma_t$]
\label{Theorem1.1}
Let $u^k=u(\sigma_t^k)(x,v,t)$, $k=1,2$ be solutions to 
the transport equation:
\begin{eqnarray*}
&& \pp_tu(x,v,t) + v\cdot \nabla u + \sigma_t^k(x,v)u
- \sigma_s(x,v)\int_V p(x,v,v')u(x,v',t) dv' = 0,\\
&& u(x,v,0) = a(x,v), \qquad x \in \Omega, \thinspace v\in V,
\thinspace k=1,2,\\
&& u = g \quad \mbox{on $\Gamma_- \times (0,T)$}.
\end{eqnarray*}
Let $u^k \in \mathcal{U}$ and 
$\Vert \sigma_t^k\Vert_{L^{\infty}(\Omega\times V)}$, 
$\Vert \sigma_s\Vert_{L^{\infty}(\Omega\times V)} \le M$.  
We assume that
$$
T > \frac{\max_{x\in\overline{\Omega}, v \in \overline{V}} (v\cdot x)
- \min_{x\in\overline{\Omega}, v \in \overline{V}} (v\cdot x)}
{\min_{v\in\overline{V}} \vert v\vert^2},
\eqno{(1.10)}
$$
and
$$
a(x,v) > 0, \quad (x,v) \in
\overline{\Omega \times V}.  \eqno{(1.11)}
$$
Then there exists a constant $C = C(M)> 0$ such that
$$
C^{-1}\left(\int^T_0\int_{\Gamma_+} (\nu(x)\cdot v)
\vert \pp_t(u^1-u^2)(x,v,t)\vert^2 dSdvdt\right)^{\frac{1}{2}}
\le \Vert \sigma_t^1 - \sigma_t^2\Vert_{L^2(\Omega\times V)}
$$
$$
\le C\left(\int^T_0 \int_{\Gamma_+} (\nu(x)\cdot v)
\vert \pp_t(u^1-u^2)(x,v,t)\vert^2 dSdvdt\right)^{\frac{1}{2}}.  \eqno{(1.12)}
$$
\end{thm}

\begin{thm}[Determination of $\sigma_s$]
\label{Theorem1.2}
Let $u^k=u(\sigma_s^k)(x,v,t)$, $k=1,2$, be the solution 
to the transport equation:
\begin{eqnarray*}
&& \pp_tu(x,v,t) + v\cdot \nabla u + \sigma_t(x,v)u
- \sigma_s^k(x,v) \int_V p(x,v,v')u(x,v',t) dv' = 0,\\
&& u(x,v,0) = a(x,v), \qquad x \in \Omega, \thinspace v\in V, \\
&& u = g \qquad \mbox{on $\Gamma_- \times (0,T)$}, \quad
k=1,2.
\end{eqnarray*}
Let $u^k \in \mathcal{U}$ and $\Vert\sigma_t\Vert_{L^{\infty}
(\Omega\times V)}$, $\Vert \sigma_s^k\Vert_{L^{\infty}(\Omega\times V)}$, 
$k=1,2$.  We assume (1.10) and (1.11).  
Then there exists a constant $C = C(M)> 0$ such that
$$
C^{-1}\left(\int^T_0\int_{\Gamma_+} (\nu(x)\cdot v)
\vert \pp_t(u^1-u^2)(x,v,t)\vert^2 dSdvdt\right)^{\frac{1}{2}}
\le \Vert \sigma_s^1 - \sigma_s^2\Vert_{L^2(\Omega\times V)}
$$
$$
\le C\left(\int^T_0\int_{\Gamma_+} (\nu(x)\cdot v)
\vert \pp_t(u^1-u^2)(x,v,t)\vert^2 dSdvdt\right)^{\frac{1}{2}}.
\eqno{(1.13)}
$$
\end{thm}

In (1.12) and (1.13), the second inequalities show the Lipschitz 
stability for the inverse problems, 
while the first inequalities are related to 
the initial/boundary value problems in which we are required to find 
$\pp_t u$ on $\Gamma_+ \times (0,T)$ for given $a$ and 
$(\sigma_t^k, \sigma_s)$, $(\sigma_t, \sigma_s^k)$, $k=1,2$.  
We obtain both-sided estimates and so the estimates 
for the inverse problems are the best possible.

For the Lipschitz stability for the inverse problems, we need the 
positivity (1.11) up to the boundary $\pp (\Omega \times V)$ 
of the initial value.  Measurements must be set up so that 
this positivity is guaranteed.  
The posivitiy condition is restricting but can be achieved in practice 
for example as follows.  
Let us consider optical tomography of the human brain 
(cf. \cite{Franceschini06,Huppert09}).  
We use a continuous-wave near-infrared laser beam and modulate the light 
by using an optical device.  
Before being temporally varied, the time-independent light is applied to 
the head.  The light is then scattered in different directions in the brain, 
and comes out.  Thus, in this setup, we can consider that the initial 
angular density $a(x,v)$ in the head is 
positive in $\overline{\Omega\times V}$.

Moreover we have to assume (1.10), that is, the observation time $T$ 
should be sufficiently large.  
%compared to the size of the domain $\Omega$.  
This is a natural condition because the transport equation has a finite 
propagation speed, which can be seen by (1.3).

In order to prove Theorems 1.1 and 1.2, it is sufficient to prove the 
linearized inverse problem below.

\begin{thm}
\label{Theorem 1.3}
We consider
$$
\pp_tu + v\cdot \nabla u + \sigma_tu
- \sigma_s\int_V p(x,v,v')u(x,v',t) dv'
= f(x,v)R(x,v,t), \quad x\in \Omega,\thinspace
v\in V, \thinspace 0 < t < T,
$$
$$
u(x,v,0) = 0, \qquad x \in \Omega, \thinspace v\in V. \\
$$
We assume
$$
R, \pp_tR \in L^2(0,T; L^{\infty}(\Omega \times V)),
\quad \sigma_t, \sigma_s \in L^{\infty}(\Omega \times V),
$$
and
$$
\pp_tu \in H^1(\Omega \times V \times (0,T)).
$$
We further assume
$$
R(x,v,0) > 0, \qquad (x,v) \in \overline{\Omega \times V}
$$
and
$$
0 < \beta < \min_{v\in\overline{V}} \vert v\vert^2, \quad
T > \frac{\max_{x\in\overline{\Omega}, v \in \overline{V}} (v\cdot x)
- \min_{x\in\overline{\Omega}, v \in \overline{V}} (v\cdot x)}{\beta}.
\eqno{(1.14)}
$$
There exist constants $C>0$ and $T>0$ such that
$$
\Vert f\Vert_{L^2(\Omega\times V)}^2 \le C\int^T_0\int_{\pp\Omega}\int_V
\vert \pp_tu\vert^2 dvdSdt  \eqno{(1.16)}
$$
for all $f \in L^2(\Omega\times V)$.
\end{thm}

\begin{thm}
\label{Theorem 1.4}
If $u = 0$ on $\Gamma_- \times (0,T)$ in Theorem 1.3, 
then there exists a constant $C>0$ such that
$$
C^{-1}\left(\int^T_0\int_{\Gamma_+} (\nu\cdot v)
\vert \pp_tu\vert^2 dSdvdt\right)^{\frac{1}{2}}
\le \Vert f\Vert_{L^2(\Omega \times V)}
\le C\left(\int^T_0\int_{\Gamma_+} (\nu\cdot v)
\vert \pp_tu\vert^2 dSdvdt\right)^{\frac{1}{2}}
\eqno{(1.17)}
$$
for any $f \in L^2(\Omega \times V)$.  
This stability estimate is the best possible.
\end{thm}

In fact, for the proof of Theorem 1.1, setting 
$u = u^1-u^2$, $f = \sigma_t^1 - \sigma_t^2$ and 
$R = - u^2$, we have the above linearized inverse problem.  
By the regularity assumption of $u^1, u^2$, we can apply Theorem 1.4 
to obtain the conclusion (1.12).  We can similarly derive 
Theorem 1.2 from Theorem 1.4.

\begin{rmk}
\label{Remark1.5}
The estimate (1.16) implies the Lipschitz stability 
$$\Vert \sigma_t^1 - \sigma_t^2\Vert_{L^2(\Omega\times V)}
\le C\left(\int^T_0 \int_{\pp\Omega}\int_V
\vert \pp_t(u^1-u^2)\vert^2 dvdSdt\right)^{\frac{1}{2}},$$
or
$$\Vert \sigma_s^1 - \sigma_s^2\Vert_{L^2(\Omega\times V)}
\le C\left(\int^T_0 \int_{\pp\Omega}\int_V
\vert \pp_t(u^1-u^2)\vert^2 dvdSdt\right)^{\frac{1}{2}}.$$
We will see below that (1.16) is obtained without assuming $u=0$ on 
$\Gamma_-\times(0,T)$, or without using 
the boundary function $g(x,v,t)$.
\end{rmk}

The article is composed of 4 sections.  In section 2, we prove (1.17).  
In section 3, we prove a key Carleman estimate 
and in section 4, we complete the proof of Theorem 1.3.

\section{Proof of Theorem 1.4}

Henceforth $C>0$ denotes generic constants which 
are independent of $f$.

\begin{lem}
\label{Lemma2.1}
Under the assumptions used in Theorem 1.3, 
there exists a constant $C>0$ such that
$$
\int^T_0\int_{\Gamma_+} \vert \pp_tu\vert^2 dvdSdt
\le C\Vert f\Vert^2_{L^2(\Omega\times V)}
+ C\int^T_0\int_{\Gamma_-} \vert \pp_tu\vert^2 dvdSdt.  \eqno{(2.1)}
$$
\end{lem}

Theorem 1.4 is obtained from Lemma 2.1 and Theorem 1.3.  
If $u=0$ on $\Gamma_-\times(0,T)$, then we have a both-sided estimate
$$
C^{-1}\int^T_0\int_{\Gamma_+}\vert \pp_tu\vert^2 dSdt
\le \Vert f\Vert_{L^2(\Omega\times V)}^2 \le C\int^T_0\int_{\Gamma_+}
\vert \pp_tu\vert^2 dSdt.
$$
This proves Theorem 1.4.
Thus the rest part of this article is devoted to the proofs of
Lemma 2.1 and Theorem 1.3.

\begin{proof}[Proof of Lemma 2.1.]
Multiplying $\pp_tu + v\cdot\nabla u + \sigma_tu
- \sigma_s\int_V pu dv' = fR$ by 
$2u$ and integrating over $\Omega\times V$, we have
\begin{eqnarray*}
&& \pp_t\int_{\Omega}\int_V \vert u(x,v,t)\vert^2 dvdx
+ \int_{\Omega}\int_V v\cdot\nabla(\vert u\vert^2) dvdx
+ 2\int_{\Omega}\int_V \sigma_tu^2 dvdx\\
-&& 2\int_{\Omega}\int_V \sigma_s(x,v)\left(\int_V p(x,v',v)u(x,v',t)dv'
\right)u(x,v,t) dvdx  = 2\int_{\Omega}\int_V fRu dvdx.
\end{eqnarray*}
By setting $E(t) = \int_{\Omega}\int_V \vert u(x,v,t)\vert^2 dvdx$, 
we obtain
\begin{eqnarray*}
&& E'(t) = - \int_{\pp\Omega}\int_V (v\cdot \nu) \vert u\vert^2 dvdS
- 2\int_{\Omega}\int_V \sigma_tu^2 dvdx \\
+ &&2\int_{\Omega}\int_V \sigma_s(x,v)\left(
\int_V p(x,v',v)u(x,v',t)dv'\right) u(x,v,t) dvdx
+ 2\int_{\Omega}\int_V fRu dvdx.
\end{eqnarray*}
Therefore, noting that $2\int_{\Omega}\int_V \vert fRu\vert dvdx \le
\int_{\Omega}\int_V \vert f\vert^2\vert R\vert^2 dvdx
+ \int_{\Omega}\int_V \vert u\vert^2 dvdx$, 
we have
\begin{eqnarray*}
\fl
E(t) - E(0) = -\int^t_0\left( \int_{\Gamma_+} + \int_{\Gamma_-}\right)
(v\cdot \nu) \vert u\vert^2 dvdS dt
- 2\int^t_0\int_{\Omega}\int_V \sigma_tu^2 dvdxdt
\\
+  2\int^t_0\int_{\Omega}\int_V \sigma_s(x,v)\left(
\int_V p(x,v',v)u(x,v',t)dv'
\right)u(x,v,t) dvdxdt
\\
+ 2\int^t_0\int_{\Omega}\int_V fRu dvdxdt
\end{eqnarray*}
$$
\le -\int^t_0 \int_{\Gamma_-}(v\cdot \nu) \vert u\vert^2 dvdS dt
+ C\int^t_0 E(\eta) d\eta + C\Vert f\Vert^2_{L^2(\Omega\times V)}
\eqno{(2.2)}
$$
for $0 \le t \le T$.  
Here by the Cauchy-Schwarz inequality, we used also
\begin{eqnarray*}
&& \left\vert \int^t_0\int_{\Omega}\int_V
\sigma_s(x,v)\left( \int_V \vert p(x,v,v')u(x,v',t)\vert dv'\right)
u(x,v,t) dvdxdt\right\vert\\
\le&& C\int^t_0 \int_{\Omega}
\left(\int_V \left(\int_V \vert u(x,v',t)\vert dv'
\right) \vert u(x,v,t)\vert dv\right) dxdt\\
\le && C\int^t_0\int_{\Omega}
\left(\left(\int_V \vert u(x,v',t)\vert^2 dv' \right)^{\frac{1}{2}}
\vert V\vert^{\frac{1}{2}}\right)
\left(\left(\int_V \vert u(x,v,t)\vert^2 dv \right)^{\frac{1}{2}}
\vert V\vert^{\frac{1}{2}}\right) dxdt\\
=&& C\vert V\vert
\int^t_0\int_{\Omega}\int_V \vert u(x,v,t)\vert^2 dv dxdt.
\end{eqnarray*}
Hence
$$
E(t) \le E(0) + \int^t_0 \int_{\Gamma_-} \vert u\vert^2 dvdS dt
+ C\Vert f\Vert^2_{L^2(\Omega\times V)} + C\int^t_0 E(\eta) d\eta, \quad
0 \le t \le T.
$$
The Gronwall inequality implies
$$
E(t) \le C\left( E(0) + \int^T_0 \int_{\Gamma_-} \vert u\vert^2 dvdS dt
+ \Vert f\Vert^2_{L^2(\Omega\times V)}\right), \quad 0 \le t \le T.
                                          \eqno{(2.3)}
$$
By (2.2), we have
\begin{eqnarray*}
&& \int^T_0 \int_{\Gamma_+} (v\cdot \nu) \vert u\vert^2 dvdS dt + E(t)\\
=&& -\int^T_0 \int_{\Gamma_-} (v\cdot \nu) \vert u\vert^2 dvdS dt
- 2\int^T_0 \int_{\Omega}\int_V \sigma_tu^2 dvdxdt
+ 2\int^T_0\int_{\Omega}\int_V fRu dvdxdt\\
+&& 2\int^T_0\int_{\Omega}\int_V \sigma_s(x,v)
\left(\int_V p(x,v',v)u(x,v',t)dv'
\right)u(x,v,t) dvdxdt\\
\le && -\int^T_0 \int_{\Gamma_-} (v\cdot \nu) \vert u\vert^2 dvdS dt
+ C\int^T_0 E(\eta) d\eta + C\Vert f\Vert^2_{L^2(\Omega\times V)}.
\end{eqnarray*}
Applying (2.3), we obtain (2.1).
\end{proof}

\section{Carleman estimate}

In this section, we prove a Carleman estimate for the proof of
Theorem 1.3.

We set
$$
Q = \Omega \times V.
$$
and
$$
Pu(x,v,t) = \pp_tu(x,v,t) + v\cdot\nabla u(x,v,t) + \sigma_tu(x,v)u(x,v,t),
\quad (x,t) \in Q, \thinspace v\in V.
$$
We set
$$
\varphi(x,t) = -\beta t + (v\cdot x)
$$
where $0 < \beta < \min_{v\in\overline{V}}\vert v\vert^2$ and
$$
B := \pp_t\varphi + (v\cdot \nabla\varphi) = -\beta + \vert v \vert^2 > 0.
$$

\begin{lem}
\label{Lemma3.1.}
There exist constants $s_0>0$ and $C>0$ such that
\begin{eqnarray*}
&& s\int_{\Omega}\int_V \vert u(x,v,0)\vert^2 e^{2s\varphi(x,0)} dvdx
+ s^2\int_Q \int_V \vert u(x,v,t)\vert^2e^{2s\varphi} dvdxdt\\
\le&& C\int_Q\int_V \vert Pu\vert^2 \weight dvdxdt
+ s\int^T_0\int_{\Gamma_+} (v\cdot \nu) \vert u\vert^2 \weight dvdSdt
\end{eqnarray*}
for all $s \ge s_0$ and $u \in H^1(\Omega\times V \times (0,T))$ satisfying 
$u(\cdot,\cdot,T) = 0$ in $\Omega\times V$.
\end{lem}

\begin{proof}
By $\sigma_t \in L^{\infty}(\Omega\times V)$, by choosing $s>0$ large, 
it suffices to prove the inequality for $\sigma_t=0$.  For any fixed 
$v \in V$, we 
set $w(x,t)=e^{s\varphi(x,t)}u(x,v,t)$ and $(L w)(x,t) = e^{s\varphi(x,t)}P
(e^{-s\varphi}w)$.  
Then
$$
Lw = \{\pp_tw + (v\cdot \nabla w)\} - sBw.
$$
Hence by $u(\cdot,\cdot, T) = 0$, we have
\begin{eqnarray*}
&& \int_Q \vert Pu\vert^2 \weight dxdt
= \int_Q \vert Lw\vert^2 dxdt\\
=&& \int_Q \vert \pp_tw + (v\cdot \nabla w)\vert^2 dxdt
+ \int_Q \vert sB\vert^2 w^2 dxdt
- 2s\int_Q B(\pp_tw + (v\cdot\nabla w)) dxdt\\
\ge &&-2s \int_Q B(\pp_tw + v\cdot \nabla w)w dxdt
+ s^2\int_Q B^2w^2 dxdt\\
= && -s \int_Q (B\pp_t(w^2) + Bv\cdot\nabla(w^2))dxdt
+ s^2\int_Q B^2w^2 dxdt\\
=&& s \int_{\Omega}  B\vert w(x,0)\vert^2 dx
- s\int^T_0\int_{\pp\Omega} B(\nu\cdot v) w^2 dSdt
+ s^2\int_Q B^2w^2 dxdt \\
\ge&& s \int_{\Omega}  B\vert w(x,0)\vert^2 dx
- s\int^T_0\int_{\pp\Omega\cap\{(v\cdot\nu)\ge 0\}} B(\nu\cdot v) w^2 dSdt
+ s^2\int_Q B^2w^2 dxdt.
\end{eqnarray*}
Substituting $w=e^{s\varphi}u$ and noting $B>0$, we have
\begin{eqnarray*}
\fl
\int_{\Omega} s\vert u(x,v,0)\vert^2 e^{2s\varphi(x,0)} dx
+ s^2\int_Q \vert u(x,v,t)\vert^2 e^{2s\varphi} dxdt
- s\int^T_0 \int_{\pp\Omega\cap\{(v\cdot\nu)\ge 0\}}
\vert u(x,v,t)\vert^2 e^{2s\varphi} dSdt
\\
\le
C\int_Q \vert Pu(x,v,t)\vert^2 e^{2s\varphi(x,t)} dxdt,
\end{eqnarray*}
where $C$ is a constant.  Integrating over $V$, we complete the proof.
\end{proof}

\begin{lem}
\label{Lemma3.2}
There exist constants $s_0>0$ and $C>0$ such that
\begin{eqnarray*}
&& s\int_{\Omega}\int_V \vert u(x,v,0)\vert^2 e^{2s\varphi(x,0)} dvdx
+ s^2\int_Q \int_V \vert u(x,v,t)\vert^2e^{2s\varphi} dvdxdt\\
\le&& C\int_Q\int_V \left\vert
\pp_tu + v\cdot \nabla u + \sigma_tu
- \sigma_s\int_V pu dv' \right\vert^2 \weight dvdxdt\\
+ && C\int^T_0\int_{\Gamma_+} \vert u\vert^2 \weight dvdSdt
\end{eqnarray*}
for all $s \ge s_0$ and $u \in H^1(\Omega\times V \times (0,T))$ satisfying 
$u(\cdot,\cdot,T) = 0$ in $\Omega\times V$.
\end{lem}

\begin{proof}
Let $C$ denote a generic constant.  Note that, for 
$\sigma_s, \sigma_t \in L^2(\Omega\times V)$, 
\begin{eqnarray*}
\fl
\int_Q\int_V \left\vert \int_V
p(x,v,v') u(x,v',t) dv' \right\vert^2 e^{2s\varphi} dvdxdt
\le C\int_Q\int_V\left(\int_V \vert u(x,v',t)\vert^2 dv' \right)
\weight dvdxdt
\\
\le
C \int_Q\int_V \vert u(x,v,t)\vert^2 \weight dvdxdt.
\end{eqnarray*}
Therefore we have
$$
\left\vert\sigma_tu-\sigma_s\int_V pu dv'\right\vert^2\le
C\vert u\vert^2.
$$
for all $(x,v)\in Q$, $v\in V$.  Thus the lemma follows from Lemma 3.1.
\end{proof}

\section{Proof of Theorem 1.3}

Henceforth $C>0$ denotes generic constants which are independent of $s>0$.  
Let $\varphi(x,t) = -\beta t + (v \cdot x)$ for $(x,t) \in Q$.  
We set
$$
R = \max_{x\in\overline{\Omega}, v\in \overline{V}} (v\cdot x), \quad
r = \min_{x\in\overline{\Omega}, v\in \overline{V}} (v\cdot x).
$$
By the conditions on $\beta > 0$ and $T > 0$, we have
$$
R - \beta T < r.  \eqno{(4.1)}
$$
Then
$$
\varphi(x,T)\le R - \beta T < r\le\varphi(x,0),\quad 
(x,v)\in\overline{\Omega}\times\overline{V}.
$$
Therefore there exist $\delta>0$ and $r_0, r_1$ such that
$R - \beta T < r_0 < r_1 < r$,
$$
\varphi(x,t) > r_1, \quad (x,v) \in\overline{\Omega}\times\overline{V},
\thinspace 0 \le t\le \delta  \eqno{(4.2)}
$$
and
$$
\varphi(x,t) < r_0, \quad (x,v) \in\overline{\Omega}\times\overline{V}, 
\thinspace
T-2\delta \le t\le T.  \eqno{(4.3)}
$$
For applying Lemma 3.2, we need a cut-off function $\chi\in C^{\infty}_0
({\Bbb R})$ such that $0 \le \chi \le 1$ and
$$
\chi(t)  =
\left\{
\begin{array}{rl}
1, \qquad & 0 \le t \le T-2\delta,\\
0, \qquad & T-\delta\le t \le T.
\end{array}\right.
\eqno{(4.4)}
$$
We set
$$
z(x,v,t) = (\pp_tu(x,v,t))\chi(t).
$$
Then $z(x,v,T) = 0$ and
$$
Pz - \sigma_s\int_Vp(x,v,v')zdv' = \chi f(\pp_tR) + (\pp_t\chi)\pp_tu,
\quad (x,t) \in Q, \thinspace v \in V
$$
and
$$
z(x,v,0) = f(x,v)R(x,v,0), \qquad x\in\Omega, \thinspace v\in V.
$$
Applying Lemma 3.2 to $z$, we obtain
$$
s\int_{\Omega}\int_V \vert z(x,v,0)\vert e^{2s\varphi(x,0)} dvdx
\le C\int_Q\int_V \vert \chi f(\pp_tR)\vert^2 \weight dvdxdt  \eqno{(4.5)}
$$
$$
+ C\int_Q\int_V \vert (\pp_t\chi)\pp_tu\vert^2 \weight dvdxdt
+ Ce^{Cs}d^2.
$$
Here                                                                           
$$
d = \int^T_0\int_{\Gamma_+} \vert \pp_tu\vert^2 dvdSdt.
$$
Since $\pp_t\chi = 0$ for $0 \le t\le T-2\delta$ or $T-\delta
\le t \le T$, by (4.3) we have
$$
\int_Q \int_V \vert (\pp_t\chi)\pp_tu\vert^2 \weight dvdxdt
=
\int^{T-\delta}_{T-2\delta} \int_{\Omega}\int_V
\vert (\pp_t\chi)\pp_tu\vert^2 \weight dvdxdt
$$
$$
\le
Ce^{2sr_0}\int^{T-\delta}_{T-2\delta} \int_{\Omega}\int_V
\vert \pp_tu\vert^2 dvdxdt.  \eqno{(4.6)}
$$
Applying (2.3) to $\pp_tu$, we obtain
$$
\int_{\Omega}\int_V \vert \pp_tu(x,v,t)\vert^2 dvdx
\le C\Vert fR\Vert^2_{L^2(\Omega\times V \times (0,T))}
+ C\int^T_0\int_{\Gamma_-} \vert \pp_tu\vert^2 dvdSdt
$$
for $0 \le t \le T$.  Therefore we can estimate (4.6) as
$$
\int_Q \int_V \vert (\pp_t\chi)\pp_tu\vert^2 \weight dvdxdt
\le Ce^{2sr_0}\Vert f\Vert^2_{L^2(\Omega\times V)}
+ Ce^{2sr_0}\int^T_0\int_{\Gamma_-} \vert \pp_tu\vert^2 dvdSdt
$$

Moreover $R(x,v,0) \ne 0$ and $z(x,v, 0) = f(x,v)R(x,v,0)$ 
for $(x,v) \in \overline{\Omega}\times\overline{V}$, we have
$$
\int_{\Omega}\int_V  \vert z(x,v,0)\vert e^{2s\varphi(x,0)} dvdx
\ge C\int_{\Omega}\int_V \vert f(x,v)\vert^2 e^{2s\varphi(x,0)} dvdx.
$$
Therefore (4.5) yields
\begin{eqnarray*}
\fl
s\int_{\Omega}\int_V \vert f(x,v)\vert^2 e^{2s\varphi(x,0)} dvdx
&\le&
C\int_Q \int_V \vert f(x,v)\vert^2 e^{2s\varphi(x,t)} dv dxdt
\\
&+&
Ce^{2sr_0}\Vert f\Vert^2_{L^2(\Omega\times V)}
+ Ce^{2sr_0}\int^T_0\int_{\Gamma_-} \vert \pp_tu\vert^2 dvdSdt
\\
&+&
Ce^{Cs}d^2.
\end{eqnarray*}
Since $\varphi(x,t) \le \varphi(x,0)$ for $(x,t) \in Q$, we have
\begin{eqnarray*}
\fl
s\int_{\Omega}\int_V \vert f(x,v)\vert^2 e^{2s\varphi(x,0)} dvdx
&\le&
C\int^T_0\int_{\Omega}\int_V \vert f(x,v)\vert^2 e^{2s\varphi(x,0)} dvdxdt
\\
&+&
Ce^{2sr_0}\Vert f\Vert^2_{L^2(\Omega\times V)}
+ Ce^{2sr_0}\int^T_0\int_{\Gamma_-} \vert \pp_tu\vert^2 dvdSdt
\\
&+&
Ce^{Cs}d^2,
\end{eqnarray*}
that is,
\begin{eqnarray*}
\fl
(s-CT)\int_{\Omega}\int_V \vert f(x,v)\vert^2 e^{2s\varphi(x,0)} dvdx
&\le&
Ce^{2sr_0}\Vert f\Vert^2_{L^2(\Omega\times V)}
\\
&+&
Ce^{2sr_0}\int^T_0\int_{\Gamma_-} \vert \pp_tu\vert^2 dvdSdt
+ Ce^{Cs}d^2
\end{eqnarray*}
for all large $s>0$.  Using $\varphi(x,0) > r_1$ and choosing $s>0$ large,
we obtain
\begin{eqnarray*}
se^{2sr_1}\int_{\Omega}\int_V \vert f(x,v)\vert^2 dvdx
&\le&
Ce^{2sr_0}\Vert f\Vert^2_{L^2(\Omega\times V)}
\\
&+&
Ce^{Cs}\int^T_0\int_{\Gamma_-} \vert \pp_tu\vert^2 dvdSdt
+ Ce^{Cs}d^2.
\end{eqnarray*}
That is,
$$
\Vert f\Vert^2_{L^2(\Omega\times V)}
\le Ce^{-2s\mu}\Vert f\Vert^2_{L^2(\Omega\times V)}
+ Ce^{Cs}\int^T_0\int_{\pp\Omega}\int_V \vert \pp_tu\vert^2 dvdSdt,
$$
for all large $s > 0$.  Here we set $\mu := r_1 - r_0 > 0$.  
Choosing $s>0$ large, we can absorb the first term on the right-hand side
into the left-hand side, and complete the proof.
\qed

\begin{rmk}
\label{Remark4.1}
If we assume 
$\Vert \pp_tu\Vert_{L^2(\Omega \times V \times (0,T))} \le M$ 
with fixed constant $M>0$, the estimate in (4.6) is written as
$$
\int_Q \int_V \vert (\pp_t\chi)\pp_tu\vert^2 \weight dvdxdt\le Ce^{2sr_0}M^2.
$$
Then $f$ is estimated less sharply but more easily without using (2.3).  
We obtain
$$
\Vert f\Vert^2_{L^2(\Omega\times V)}\le CM^2e^{-2s\mu}+Ce^{Cs}d^2.
$$
By minimizing the right-hand side with respect to $s$, the H\"{o}lder 
stability is obtained.  
That is, there exist constants $\theta \in (0,1)$, $C>0$ and $T>0$ such that
$$
\Vert f\Vert_{L^2(\Omega\times V)}^2 \le C\left(\int^T_0\int_{\Gamma_+}
\vert \pp_tu\vert^2 dvdSdt\right)^{\theta}  \eqno{(1.15)}
$$
for all $f \in L^2(\Omega\times V)$.  The use of (2.3) is needed to 
obtain the Lipschitz stability.
\end{rmk}

%\ack

%\setcounter{section}{1}
%\appendix

\end{document}